\documentclass[12pt]{article}
\usepackage{amsmath}
\usepackage{amsthm}
\usepackage{amsfonts}
\usepackage{textcomp}
\usepackage[margin=1.25 in]{geometry}
\usepackage{setspace}
\begin{document}
\title{An upper bound on the sum of signs of permutations with a condition on their prefix sets}
\author{Nikola Djoki\'c\\
ETH Z\"urich}
\date{December 2, 2013}
\maketitle

\newtheorem{thr}{Theorem}
\newtheorem{lemma}{Lemma}
\newtheorem{prop}{Proposition}
The following Proposition is a positive answer to a question about cancellations
between permutations that arises in a model problem in the many body theory of Fermions. 
It concerns the mathematically rigorous implementation of 
the Pauli exclusion principle. See [1], Question 4.1.
\begin{prop}
Let $(x_1, x_2,\ldots,x_n)$ and $(y_1, y_2,\ldots,y_n)$ be two tuples of real numbers.
Give a weight $\epsilon_\sigma$ to every permutation $\sigma$ of $\{1,2,\ldots,n\}$ as follows:
\begin{equation}\epsilon_\sigma = \begin{cases}

0 & \text{if }   
x_{\sigma(1)} + x_{\sigma(2)} + \cdots +x_{\sigma(k)} < y_k
\text{ for some } k \in \{1,\ldots,n\}
\cr
{\rm sign}(\sigma), &\text{otherwise}
\cr\end{cases}
\end{equation}
Then
$$
\sum\limits_{\sigma \in S_n} \epsilon_\sigma \leq \sqrt{n}^n
$$
\end{prop}

The Propositon is a consequence of the following Theorem.

\begin{thr}
Fix a natural number $n$ and set $N=\{1,2,\ldots,n\}$. Let $f$ be any function
on the power set $\mathcal P(N)$ which takes values in $[-1,1]$ and define
$g:\mathcal P(N) \to \mathbb{R}$ by
\begin{equation}
g(T) = \sum\limits_{\sigma\in S_k} {\rm sign}(\sigma) \prod\limits_{i=1}^k 
f( \{t_{\sigma(1)},\ldots,t_{\sigma(i)}\})
\qquad  {\rm if \ }
T=\{t_1,\ldots,t_k\} \ {\rm with\ } t_1<\cdots<t_k
\end{equation}
Then 
$$
|g(N)| \le \sqrt{n}^n
$$
\end{thr}
To see that the proposition follows from the theorem, choose as $f$ the function
\begin{equation}
f_{xy}(S)=
\begin{cases}
1,&\text{if }\sum\limits_{j\in S}x_j\geq y_{|S|}\\
0,&\text{otherwise}
\end{cases}
\end{equation}
To prove the theorem, we will first need some lemmas.

\begin{lemma}\begin{equation}\forall T\in\mathcal P(N)\setminus\emptyset\colon\quad g(T)=f(T)\sum\limits_{a\in T}g(T\setminus\{a\})(-1)^{\vert\{t\in T,t>a\}\vert}\end{equation}\end{lemma}
\begin{proof} Let $k\in\mathbb{N}$, $0<k\leq n$ be given. Let $\phi\colon S_{k-1}\times\{1,\ldots,k\}\to S_k$, \begin{equation}\phi(\sigma\times a)(i)=\begin{cases}a&\text{if }i=k\\ \sigma(i)&\text{if }\sigma(i)<a\\ \sigma(i)+1&\text{if }\sigma(i)\geq a\end{cases}\end{equation}
$\phi$ is bijective and \begin{equation}\text{sign}(\phi(\sigma\times a))=\text{sign}(\sigma)(-1)^{\vert\{b\in\{1,\ldots,k\},b>a\}\vert}\end{equation}
Replacing $\sum\limits_{\sigma\in S_k}$ with $\sum\limits_{a\in\{1,\ldots,k\}}\sum\limits_{\sigma\in S_{k-1}}$ and $\sigma$ with $\phi(\sigma\times a)$ in (2) yields (4).\end{proof}

Let $e_1,\ldots,e_n$ be the canonical basis of $\mathbb{R}^n$. For $k\in\mathbb{N}$, $0\leq k\leq n$, $S\subset N$, $S=\{s_1,\ldots,s_k\}$, $s_1<\cdots<s_k$ let \begin{equation}\delta_S=e_{s_1}\wedge\cdots\wedge e_{s_k}\in\bigwedge^k\mathbb{R}^n\end{equation}\newline
Let $V=\bigoplus\limits_{k=0}^n\bigwedge^k\mathbb{R}^n$.\newline
The standard scalar product $\langle\cdot,\cdot\rangle$ on $V$ is defined such that $\{\delta_S, S\in\mathcal P(N), |S|=k\}$ is an orthonormal basis of $\bigwedge^k\mathbb{R}^n$ and $\{\delta_S, S\in\mathcal P(N)\}$ is an orthonormal basis of $V$.\newline
Let $P\in\text{End}(V)$ be the linear map with $P(\delta_S)=f(S)\delta_S$.\newline
For $k\in\mathbb{N}$, $0\leq k\leq n$, let $G_k=\{S\in\mathcal P(N),\vert S\vert=k\}$.\newline
Let $\alpha=\frac{1}{\sqrt{n}}(\delta_{\{1\}}+\cdots+\delta_{\{n\}})$.\newline
Let $R\in\text{End}(V)$, $\varphi\mapsto\varphi\wedge\alpha$ be the multiplication with $\alpha$.\newline
Identify $\bigwedge^k\mathbb{R}^n$ with $\mathbb{R}^{G_k}$ using \begin{equation}\delta_S(T)=\begin{cases}1&\text{if }S=T\\0&\text{if }S\neq T\end{cases}\end{equation}
Using this identification, we have $g\vert_{G_k}\in\bigwedge^k\mathbb{R}^n$.\newline
(4) is equivalent to \begin{equation}g\vert_{G_k}=\sqrt{n}(P(R(g\vert_{G_{k-1}})))\end{equation}

\begin{lemma}
For all $v_1,\ldots,v_k, w_1,\ldots,w_k\in\mathbb{R}^n$:
\begin{equation}
\langle v_1\wedge\cdots\wedge v_k, w_1\wedge\cdots\wedge w_k\rangle={\rm det}((\langle v_i,w_j\rangle)_{ij})
\end{equation}
\end{lemma}
\begin{proof}
For each $i$, $1\leq i\leq k$, both sides of the equation are linear in $v_i$ and in $w_i$. The equation holds when $v_1,\ldots,v_k, w_1,\ldots,w_k\in\{e_1,\ldots,e_n\}$. Therefore, it holds for all $v_1,\ldots,v_k, w_1,\ldots,w_k\in\mathbb{R}^n$.
\end{proof}

\begin{proof}[Proof of the Theorem] From (9), it follows that $g(N)\delta_N=\sqrt n^n(P\circ R)^n\delta_\emptyset$. Therefore, using submultiplicativity of the operator norm: $|g(N)|\leq\sqrt n^n(|P|_{op}|R|_{op})^n$.\newline
$|P|_{op}\leq 1$ follows directly from the definition.\newline
Let $R^*$ be the adjoint of $R$. Choose an orthonormal basis of $\mathbb{R}^n$ which contains $\alpha$ (noting that $|\alpha|=1$) and expand it to a basis of $V$ like in (7). This expanded basis is orthonormal because of (10), therefore $R^*\circ R$ is the projection onto\newline$\{\varphi\in V,\varphi\wedge\alpha=0\}^\bot$, therefore $|R|_{op}\leq 1$.
\end{proof}

\noindent
{\it References:}
\vskip .2cm
\noindent
[1] J.Feldman, H.Kn\"orrer, E.Trubowitz: Construction of a $2-d$ Fermi
Liquid.
\noindent
In: XIV. International Congress on Mathematical Physics, pp.245-260.
Editor: J.-C. Zambrini.
World Scientific 2005

\end{document}